\documentclass[11pt]{amsart}
 \usepackage{amsfonts,amssymb,amsmath,mathrsfs}

\usepackage{color,latexsym,amsfonts}
\newcommand{\blue}[1]{\textcolor{blue}{#1}}

\numberwithin{equation}{section}

 \def\qed{\hfill$\Box$\medskip}
 \newtheorem{theorem}{Theorem}[section]
 \newtheorem{lemma}[theorem]{Lemma}

 \newtheorem{remark}[theorem]{Remark}

 \def\<{\langle}\def\>{\rangle}

 \def\blue{\color{blue}}
 \def\beqlb{\begin{eqnarray}}\def\eeqlb{\end{eqnarray}}
 \def\beqnn{\begin{eqnarray*}}\def\eeqnn{\end{eqnarray*}}

 \def\<{\langle}\def\>{\rangle}

 \def\blue{\color{blue}}
 \def\beqlb{\begin{eqnarray}}\def\eeqlb{\end{eqnarray}}
 \def\beqnn{\begin{eqnarray*}}\def\eeqnn{\end{eqnarray*}}

\begin{document}

\noindent{(20170325)}

\bigskip\bigskip

\centerline{\LARGE\bf The Minimal Position of  }
\medskip

\centerline{\LARGE\bf a Stable Branching Random Walk}

\bigskip\bigskip

\centerline{Jingning Liu and Mei
Zhang\,\footnote{Corresponding author. Supported by NSFC
(11371061)}}

\bigskip\bigskip

\centerline{School of Mathematical Sciences  }
 \centerline{Laboratory of Mathematics and Complex Systems }
 \centerline{ Beijing Normal University }

 \centerline{Beijing 100875, People's Republic of China}

 \centerline{E-mails:\;{\tt
liujingning@mail.bnu.edu.cn
  and
meizhang@bnu.edu.cn}}
\bigskip

\bigskip

\medskip

{\narrower

\noindent\textit{\bf Abstract} In this paper, a branching random walk $(V(x))$ in the boundary case is studied, where the associated one dimensional random walk is in the domain of attraction of an $\alpha-$stable law with $1<\alpha<2$. Let $M_n$ be the minimal position of $(V(x))$ at generation $n$. We established an integral test to describe the lower limit of $M_n-\frac{1}{\alpha}\log n$ and a law of iterated logarithm for the upper limit of $M_n-(1+\frac{1}{\alpha})\log n$.
\medskip

\noindent\textit{Mathematics Subject Classification (2010).} Primary
60J80; secondary 60F05.

 \smallskip

\noindent\textit{\bf Key words and phrases} branching random walk, minimal position, additive martingale, upper limit, lower limit.

\medskip

\section{introduction}

We consider a discrete-time one-dimensional branching random walk. It starts with an initial ancestor particle located at the origin. At time $1$, the particle dies, producing a certain number of new particles. These new particles are positioned according to the distribution of the point process $\Theta$. At time $2$, these particles die, each giving birth to new particles positioned (with respect to the birth place) according to the law of $\Theta$. And the process goes on with the same mechanism. We assume the particles produce new particles independently of each other. This system can be seen as a branching tree $\textbf{T}$ with the origin as the root. For each vertex $x$ on $\textbf{T}$, we denote its position by $V(x)$. The family of the random variables $(V(x))$ is usually referred as a branching random walk (Biggins~\cite{bi}). The generation of $x$ is denoted by $|x|$.

We assume throughout the remainder of the paper, including in the statements of
theorems and lemmas, that
\begin{align}
& \mathbf{E}\Big(\sum_{|x|=1}1\Big)>1, \;\; \mathbf{E}\Big(\sum_{|x|=1}e^{-V(x)}\Big)=1,\;\; \mathbf{E}\Big(\sum_{|x|=1}V(x)e^{-V(x)}\Big)=0. \label{stable0}&
\end{align}
Condition (\ref{stable0})  means that the branching random walk $(V(x))$ is supercritical and in boundary case (see for example, Biggins and Kyprianou~\cite{BK05}). Every branching random walk satisfying certain mild integrability assumptions can be reduced to this case by some renormalization; see Jaffuel~\cite{Ja} for more details.

Denote $M_n=\min_{|x|=n}V(x)$, i.e.,  the minimal position at generation $n$. We introduce the conditional probability $\mathbf{P}^*(\cdot):=\mathbf{P}(\cdot\,|\mbox{non-extinction}).$ Under (\ref{stable0}), $M_n\rightarrow\infty, \mathbf{P}^*\mbox{--a.s.}$ (See for example, \cite{B77} and \cite{Ly}). The asymptotic behaviors of $M_n$ have been extensively studied in \cite{AR09}, \cite{A13}, \cite{BZ06}, \cite{HS09}, etc. In particular, under (\ref{stable0}) and certain exponential integrability conditions, Hu and Shi~\cite{HS09} obtained the following:
\begin{align}
\limsup_{n\to \infty}\,\frac{M_n}{\log n}=\frac{3}{2}, \quad \mathbf{P}^*\mbox{--a.s.}\nonumber\\
\liminf_{n\to \infty}\,\frac{M_n}{\log n}=\frac{1}{2}, \quad \mathbf{P}^*\mbox{--a.s.}\nonumber
\end{align}
It showed that there is a phenomena of fluctuation at the logarithmic scale. Aidekon~\cite{A13} proved the convergence in law of $M_n-\frac{3}{2}\log n$ when (\ref{stable0}) and the following two conditions hold:
\begin{align}
\mathbf{E}\Big(\sum_{|x|=1}V^2(x)e^{-V(x)}\Big)<\infty,\label{incon}
\end{align}
\begin{align}
\mathbf{E}\big(X(\log_{+}{X})^{2}+\widetilde{X}(\log_{+}{\widetilde{X}})\big)<\infty, \label{incon1}
\end{align}
where $X:=\sum_{|x|=1}e^{-V(x)}$, $\widetilde{X}:=\sum_{|x|=1}V(x)_+e^{-V(x)}$, and $V(x)_+:=\max\{V(x),0\}$.
Later, Aidekon and Shi~\cite{AS14} proved that under \eqref{stable0}--\eqref{incon1},
$$
\liminf_{n\rightarrow\infty}\,\big(M_n-\frac{1}{2}\log n\big)=-\infty,\quad \mathbf{P}^*\mbox{--a.s.}
$$
Based on this result, Hu~\cite{H15} established the second order limit under the same assumptions \eqref{stable0}--\eqref{incon1}:
$$
\liminf_{n\rightarrow\infty} \,\frac{M_n-\frac{1}{2}\log n}{\log\log n}=-1,\quad \mathbf{P}^*\mbox{--a.s.}
$$
When (\ref{stable0}), (\ref{incon1}) and a higher order integrability condition for $V(x)$ hold \Big(i.e. $\mathbf{E}\big(\sum_{|x|=1}$ $(V(x)_+)^3e^{-V(x)}\big)<\infty$\Big),  the upper limit was  obtained by Hu ~\cite{H13}:
$$
\limsup_{n\rightarrow\infty}\frac{M_n-\frac{3}{2}\log n}{\log\log\log n}=1.
$$

Throughout the following, $c, c', c_1,c_2,\cdots$ will denote some positive constants whose value may change from place to place. $f(x)\sim g(x)$ as $x\to \infty$ means that $\lim_{x\to \infty} \frac{f(x)}{g(x)}=1$; $f(x)=O(g(x))$ as $x\to \infty$ means that $\lim_{ x\to \infty} \frac{f(x)}{g(x)}=c$.

In this paper, we shall consider the random walk by assuming that
\begin{align}
&  \mathbf{E}\,\Big(\sum_{|x|=1}\mathbf{1}_{\{V(x)\leq -y\}} e^{-V(x)}\Big)=O({y^{-\alpha-\varepsilon})},\quad y\to \infty; \label{stable1}&\\
&  \mathbf{E}\,\Big(\sum_{|x|=1}\mathbf{1}_{\{V(x)\geq y\}} e^{-V(x)}\Big)\sim \frac{c}{y^{\alpha}},\quad y\to \infty; \label{stable2}&\\
& \mathbf{E}\big(X(\log_{+}{X})^{\alpha}+\widetilde{X}(\log_{+}{\widetilde{X}})^{\alpha-1}\big)<\infty,  \label{stable3}&
\end{align}
where $\alpha\in (1,2)$, $\varepsilon>0$, $c>0$. Under   (\ref{stable1}) and (\ref{stable2}), in Section 2, we shall see that there is one-dimensional random walk $\{S_n\}$ corresponding to $(V(x))$, where
$S_1$ belongs to the domain of attraction of a spectrally positive stable law. We call $(V(x))$ a {\it stable random walk}.
we  shall study the asymptotic behavior of $M_n$  for the stable random walk $(V(x))$ under the conditions \eqref{stable0}, \eqref{stable1}--\eqref{stable3}.  Our main results are the following Theorems~\ref{T:1.1}--\ref{T:1.4}.

\begin{theorem}\label{T:1.1}
Assume \eqref{stable0}, \eqref{stable1}--\eqref{stable3}. For any nondecreasing function $f$ satisfying \\
$\lim_{x\rightarrow\infty}f(x)=\infty$, we have
\begin{align}
{\mathbf{P}^*(M_n-\frac{1}{\alpha}\log n<-f(n), \quad\mbox{i.o.})=\left\{
\begin{aligned}
&0\\
&1\\
\end{aligned}
\right.\;\;\Leftrightarrow\;\; \int_0^\infty \frac{1}{\;t\,e^{f(t)}\;}d\,t\left\{
\begin{aligned}
&<\infty\\
&=\infty. \\
\end{aligned}
\right.}
\end{align}
\end{theorem}

The behavior of the minimal position $M_n$ is closely related to the so-called additive martingale $(W_n)_{n\geq0}$:
\begin{align}
W_n:=\sum_{|u|=n}e^{-V(u)}, n\geq0. \nonumber
\end{align}
By \cite{bi2} and \cite{Ly}, $W_n\rightarrow0$ almost surely as $n\rightarrow\infty$. A similar integral test for the upper limits of $W_n$ can be described as follows:
\begin{theorem} \label{T:1.2}
Assume  \eqref{stable0}, \eqref{stable1}--\eqref{stable3}. For any nondecreasing function $f$ satisfying\\
 $
\lim_{x\rightarrow\infty}f(x)=\infty$, we have $\mathbf{P}^*\mbox{--a.s.}$
\begin{align}
\limsup_{n\rightarrow\infty} \frac{n^{\frac{1}{\alpha}}W_n}{f(n)}=\left\{
\begin{aligned}
&0\\
&\infty\\
\end{aligned}
\right.\;\;\Leftrightarrow\;\; \int_0^\infty \frac{1}{\;t\,f(t)\;}d\,t\left\{
\begin{aligned}
&<\infty\\
&=\infty. \\
\end{aligned}
\right.
\end{align}
\end{theorem}


\begin{theorem}\label{T:1.3}
Assume \eqref{stable0}, \eqref{stable1}--\eqref{stable3}. We have \begin{align}
\liminf_{n\rightarrow\infty}\frac{M_n-\frac{1}{\alpha}\log n}{\log\log n}=-1, \; \mathbf{P}^*\mbox{--a.s.}
\end{align}
\end{theorem}

\begin{remark}
 For the random walk $(V(x))$ satisfying (\ref{stable0})--(\ref{incon1}), where the one-dimensional random walk associated with $(V(x))$ has finite variance, Hu \cite[Theorem 1.1, 1.2, Proposition 1.3]{H15}  established the corresponding theorems for $M_n$ and $W_n$. Theorem \ref{T:1.1}--\ref{T:1.3} are extensions of them for the stable random walk under \eqref{stable0}, \eqref{stable1}--\eqref{stable3}. Now the one-dimensional random walk $\{S_n\}$  associated with $(V(x))$ has no finite variance (see Section 2 for details).
\end{remark}

\begin{theorem}\label{T:1.4}
Assume \eqref{stable0}, \eqref{stable1}--\eqref{stable3}. We have
\begin{align}
\limsup_{n\rightarrow\infty}\frac{M_n-(1+\frac{1}{\alpha})\log n}{\log\log\log n}\geq 1, \; \mathbf{P}^*\mbox{--a.s.} \nonumber
\end{align}
\end{theorem}

\begin{remark}

The upper limit for $M_n$ is established in Hu~\cite{H13} under  (\ref{stable0})--(\ref{incon1}) and finite third order moment $$\mathbf{E}\Big(\sum_{|x|=1}(V(x)_+)^3e^{-V(x)}\Big)<\infty. $$  While in this paper, under \eqref{stable0}, \eqref{stable1}--\eqref{stable3}, for $k\ge \alpha$, $(V(x))$ no longer satisfies the integrability condition  $$\mathbf{E}\Big(\sum_{|x|=1}(V(x)_+)^ke^{-V(x)}\Big)<\infty. $$  In this case, we have not got the condition for the upper limit of $\frac{M_n-(1+\frac{1}{\alpha})\log n}{\log\log\log n}$. In Theorem~\ref{T:1.4},  only a lower bound  is obtained.
\end{remark}

\section{Stable random walk}

In this section,  we first introduce an one-dimensional random walk associated with the branching random walk.

For $a\in\mathbb{R}$, we denote by $\mathbf{P}_a$ the probability distribution associated to the branching random walk $(V(x))$ starting from $a$, and $\mathbf{E}_a$ the corresponding expectation. For any vertex $x$ on the tree $\bf T$, we denote the shortest path from the root $\varnothing$ to $x$ by $\langle\varnothing,x\rangle\!:=\!\{x_0,x_1,x_2,\ldots,x_{|x|}\}$. Here $x_i$ is the ancestor of $x$ at the $i$-th generation. For any $\mu, \nu\in \textbf{T} $, we use the partial order $\mu<\nu$ if $\mu$ is an ancestor of $\nu$.  Under \eqref{stable0}, there exists a  sequence of independently and identically distributed real-valued random variables $S_1,S_2-S_1,S_3-S_2,\ldots,$ such that for any $n\geq1, a\in\mathbb{R}$ and any measurable function $g:\mathbb{R}^n\rightarrow[0,\infty),$
\begin{align}\label{manytoone}
\mathbf{E}_a\Big(\sum_{|x|=n}g\big(V(x_1),\ldots,V(x_n)\big)\Big)=\mathbf{E}_a\Big(e^{S_n-a}g(S_1,\ldots,S_n)\Big),
\end{align}
where, under $\mathbf{P}_a$, we have $S_0=a$ almost surely.  (\ref{manytoone}) is called the {\it  many-to-one formula}.  We will write $\mathbf{P}$ and $\mathbf{E}$ instead of $\mathbf{P}_0$ and $\mathbf{E}_0$. Since $\mathbf{E}\big(\sum_{|x|=1}V(x)e^{-V(x)}\big)=0$, we have $\mathbf{E}(S_1)=0$. By  (\ref{stable1}) and (\ref{stable2}), it is not difficult to see that $\mathbf{E}S_1^k=\infty$ for $k\ge \alpha$.
Under conditions \eqref{stable1} and \eqref{stable2},  $S_1$ belongs to the domain of attraction of a spectrally positive stable law with characteristic function
\begin{align}
  G(t) :=\exp\big\{-c_0|t|^\alpha\big(1-i\frac{t}{|t|}\tan{\frac{\pi\alpha}{2}}\big)\big\},\quad c_0>0. \nonumber
\end{align}

The following are some estimates on $(S_n)$, which are key in the proofs of the main theorems.
\begin{lemma}
Let $0<\lambda<1$. There exist positive constants $c_1,c_2,\cdots, c_5$ such that for any $a\ge 0, b\ge -a, 0\le u\le v$ and $n\ge1$,
\begin{align}
&\mathbf{P}(\underline{S}_n\geq -a) \leq c_1 \frac{(1+a)}{\,n^{\frac{1}{\alpha}}},\label{S:1}\\
&\mathbf{P}(\underline{-S}_n\geq -a) \leq c_2 \frac{(1+a)^{\alpha-1}}{\,n^{1-\frac{1}{\alpha}}},\ \label{S:2}\\
&\mathbf{P}(S_n\leq b,\,\underline{S}_n\geq-a)\leq c_3\frac{\;(1+a)(1+a+b)^\alpha}{n^{1+\frac{1}{\alpha}}},\label{S:3}\\
& \mathbf{P}\big(\underline{S}_{\;\llcorner\lambda n\lrcorner}\geq-a,\min_{i\in[\lambda n,n]\cap \mathbb{Z}}S_i\geq b,S_n\in[b+u,b+v]\,\big)\leq  c_4 \frac{(1+v)^{\alpha\!-\!1}(1+v-u)(1+a)}{n^{1+\frac{1}{\alpha}}},\label{S:4}\\
& \mathbf{P}(\underline{S}_n\geq-a, \min_{\lambda n\leq i<n}S_i>b, S_n\leq b)\leq c_5(1+a)n^{-1-\frac{1}{\alpha}}, \label{S:5}
\end{align}
where $\underline{S}_n:=\min_{0\le i\leq n}S_i$ and $\underline{-S}_n:=\min_{0\le i\leq n}(-S_i)$.
\end{lemma}

\textbf{Proof.} The proofs of (\ref{S:1})--(\ref{S:4}) have been given in \cite[Lemmas 2.1--2.4]{sta}. Here we only prove (\ref{S:5}).
Let $f(x):=\mathbf{P}(S_1\leq -x)$. Denote the event in \eqref{S:5} by {\blue$E_{\eqref{S:5}}$}. Applying the Markov property of $(S_i)$ at $n-1$, and using \eqref{S:4} we have that
\begin{align}
\mathbf{P} (E_{\eqref{S:5}}) 
\leq & \;\sum_{j=0}^\infty f(j)\,\mathbf{P}\Big(\underline{S}_{n-1}\geq-a, \min_{\lambda n\leq i<n}S_i>b,
b+j<S_{n-1}{\blue\leq}b+j+1\Big)\nonumber\\
\leq & \;c\,(1+a)n^{-1-\frac{1}{\alpha}}\sum_{j=0}^\infty f(j) (2+j)^{\alpha-1}.\nonumber\\
\end{align}
By (\ref{stable1}), $$\sum_{j=0}^\infty f(j) (2+j)^{\alpha-1}\leq c\,\mathbf{E}\big((-S_1)^\alpha\textbf{1}_{\{S_1<0\}}\big)<\infty. $$
Then the proof is completed.
\qed
}
\section{Proofs of Theorems 1.1--1.3}

In this section, first we prove Theorem 1.1 and Theorem 1.2. Noticing that $W_n\geq e^{-M_n}$, we only need to prove the the convergence part in Theorem \ref{T:1.2}, i.e.,
\begin{align}
\int_0^\infty \frac{dt}{tf(t)}<\infty \Rightarrow \limsup_{n\to \infty} \frac{n^\frac{1}{\alpha}W_n}{f(n)}=0,\;\;\;  \mathbf{P}^*\mbox{--a.s.} \,\,; \label{conver}
\end{align}
and the divergence part in Theorem \ref{T:1.1}, i.e.,
\begin{align}
\int_0^\infty \frac{dt}{te^{f(t)}}=\infty \Rightarrow \mathbf{P}^*(M_n-\frac{1}{\alpha}\log n<-f(n),\quad \mbox{i.o.})=1. \label{div}
\end{align}

We define the set containing brothers of vertex $x$ by $\Omega(x)$, i.e., $\Omega(x)=\{y:y_{|y|-1}=x_{|x|-1},y\neq x\}$.
For $\beta\ge 0$, define $$W_n^{\beta}:=\sum_{|x|=n}e^{-V(x)}
\mathbf{1}_{\{\underline{V}(x)\geq-\beta\}},$$
 where
$\underline{V}(x):=\min_{0\leq i\leq|x|}V(x_i).$

To prove (\ref{conver}), we need the following  lemma. 

\begin{lemma}\label{L:3.1}
Assume \eqref{stable0}, \eqref{stable1}--\eqref{stable3}. For any $\beta\geq0$, there exists a constant $c$ such that for any $1<n\leq m$ and $\lambda>0$, we have
\begin{align}
\mathbf{P}(\max_{n\leq k\leq m}k^{\frac{1}{\alpha}}W^{(\beta)}_k>\lambda)\leq c~\Big(\frac{\log n}{n^{\frac{1}{\alpha}}}+\frac{1}{\lambda}(\frac{m}{n})^{\frac{1}{\alpha}}\Big).
\end{align}
\end{lemma}

\begin{proof}
We introduce another martingale related to $W^\beta_k$:
\begin{align}
W^{(\beta,n)}_k:=\sum_{|x|=k}e^{-V(x)}\mathbf{1}_{\{\underline{V}(x_n)\geq-\beta\}}, \quad n\leq k\leq m+1,\nonumber
\end{align}
where $\underline{V}(x_n):=\min_{1\le i\leq n}V(x_i)$. Therefore,
\begin{align}
\mathbf{P}(\max_{n\leq k\leq m}k^{\frac{1}{\alpha}}W^{(\beta)}_k>\lambda)&\leq \mathbf{P}(\max_{n\leq k\leq m}k^{\frac{1}{\alpha}}W^{(\beta,n)}_k>\lambda)+ \mathbf{P}(\min_{n\leq k\leq m}\min_{|x|=k}V(x)<-\beta)\nonumber.
\end{align}
By the branching property, for $n\leq k\leq m$, we have that
$$\mathbf{E}(W^{(\beta,n)}_{k+1}|\mathcal{F}_k)=W^{(\beta,n)}_k. $$
Hence, by Doob's maximal inequality,
\begin{align}
\mathbf{P}(\max_{n\leq k\leq m}k^{\frac{1}{\alpha}}W^{(\beta,n)}_k\geq\lambda)\leq\frac{m^\frac{1}{\alpha}}{\lambda}
\mathbf{E}(W^{(\beta,n)}_n)=\frac{m^\frac{1}{\alpha}}{\lambda}
\mathbf{E}(W^\beta_n).
\end{align}
From (\ref{manytoone}) and $\eqref{S:1}$, it follows that
\begin{align}
\mathbf{P}(\max_{n\leq k\leq m}k^{\frac{1}{\alpha}}W^{(\beta,n)}_k\geq\lambda)\leq c \frac{1}{\lambda}(\frac{m}{n})^{\frac{1}{\alpha}}. \label{L:3.1_1}
\end{align}
On the other hand,  By Aidekon \cite[P. 1403]{A13} we know $\mathbf{P}(\inf_{x\in\textbf{T}}V(x)<-x)\leq e^{-x}$ for $x\geq0$. Hence 
\begin{align}
&\mathbf{P}(\min_{n\leq k\leq m}\min_{|x|=k}V(x)<-\beta\,)\nonumber\\ &\leq\mathbf{P}(\inf_{x\in\textbf{T}}V(x)<-\log n)
+\mathbf{P}(\min_{n\leq k\leq m}\min_{|x|=k}V(x)<-\beta,\inf_{x\in\textbf{T}}V(x)\geq-\log n) \nonumber\\
&\leq \frac{1}{n}+\sum_{k=n}^m\mathbf{E}(\sum_{|x|=k}\mathbf{1}_{\{V(x)<-\beta,V(x_n)\geq-\beta,\ldots,V(x_{k-1})\geq-\beta,
\underline{V}(x)\geq-\log n\}}) .\nonumber
\end{align}
By (\ref{manytoone}) and \eqref{S:1},
\begin{align}
\mathbf{P}(\min_{n\leq k\leq m}\min_{|x|=k}V(x)<-\beta\,)\leq \frac{1}{n}+c \,\mathbf{P}(\underline{S}_n\geq-\log n)\leq c~\frac{\log n}{n^{\frac{1}{\alpha}}},\nonumber
\end{align}
which together with \eqref{L:3.1_1}, completes the proof.
\end{proof}
\qed

\noindent\textbf{Proof of  (\ref{conver}). }Let $n_j=2^j$. According to Lemma \ref{L:3.1}, for all large $j$ we have
\begin{align}
\mathbf{P}\big(\max_{n_j\leq k\leq n_{j+1}}k^{\frac{1}{\alpha}}W^\beta_k>f(n_j)\big)\leq c\,\bigg( \frac{\log n_j}{n_j^{\frac{1}{\alpha}}}+\frac{2^{\frac{1}{\alpha}}}{f(n_j)}\bigg). \label{F:3.6}
\end{align}
By our assumption for $f$, $$\sum_{j\geq j_0}\frac{1}{f(n_j)}\leq\sum_{j\geq j_0}\frac{1}{\log2}\int_{n_{j-1}}^{n_j}\frac{1}{f(x)x}\mathrm{d}x<\infty. $$ Hence
\begin{align*}
  \sum_{j\geq j_0} \mathbf{P}\big(\max_{n_j\leq k\leq n_{j+1}}k^{\frac{1}{\alpha}}W^\beta_k>f(n_j)\big)<\infty.
\end{align*} By Borel-Cantelli Lemma, for all large $k$,
\begin{align}
k^{\frac{1}{\alpha}}W^\beta_k\leq f(k),\quad\mathbf{P}\mbox{--a.s.} \nonumber
\end{align}
  Letting $\beta\to \infty$, we have $k^{\frac{1}{\alpha}}W_k\leq f(k),\quad\mathbf{P}\mbox{--a.s.}$ As a consequence,
\begin{align}
\limsup_{k\rightarrow\infty}\frac{k^{\frac{1}{\alpha}}W_k}{f(k)}\leq1,\quad\mathbf{P}\mbox{--a.s.} \nonumber
\end{align} 
Replacing $f$ by $\varepsilon f$, and letting $\varepsilon\rightarrow0$, we complete the proof.

\qed

  Fix $K\geq0$.  Now we define  for $n<k\leq \alpha n$,
\begin{align}
&A_k^{(n,\lambda)}:=\big\{x:|x|=k,V(x_i)\geq a_i^{(n,\lambda)}, 0\leq i\leq k,V(x)\leq\frac{1}{\alpha}\log n-\lambda+K\big\},\nonumber\\
&B_k^{(n,\lambda)}:=\big\{x:|x|=k,\sum_{u\in\Omega(x_{i+1})}(1+(V(u)-a_i^{(n,\lambda)})_+)e^{-(V(u)-a_i^{(n,\lambda)})}\leq c'e^{-b_i^{(k,n)}},\; 0\leq i\leq k\!-\!1\big\},\nonumber
\end{align}
where $
a_i^{(n,\lambda)}=\mathbf{1}_{\{\frac{\alpha}{4}n<i\leq k\}}(\frac{1}{\alpha}\log n-\lambda)\,$,
$b_i^{(k,n)}=\mathbf{1}_{\{0\leq i\leq\frac{\alpha}{4}n\}}i^{\frac{\gamma}{2}}+\mathbf{1}_{\{\frac{\alpha}{4}n<i\leq\alpha n\}}(k-i)^{\frac{\gamma}{2}}$,   $K>0$, $\gamma=\frac{1}{\alpha(\alpha\!+\!1)}$ and $c'$ is a positive  constant chosen as in \cite[Lemma 7.1]{sta}. 

Lemmas~\ref{L:3.2} and \ref{L:3.3} are preparing works for the proof of (\ref{div}).

\begin{lemma}\label{L:3.2}
Assume \eqref{stable0}, \eqref{stable1}--\eqref{stable3}.  There exist some positive constants $K$ and
$c_6, c_7$ such that for all $n\geq2$, $0\leq\lambda\leq\frac{1}{\,2\alpha}\log n$,
\begin{align}
c_6e^{-\lambda}\leq\mathbf{P}\Big(\bigcup_{k=n\!+\!1}^{\alpha n}A_k^{(n,\lambda)}\cap B_k^{(n,\lambda)}\Big)\leq c_7\,e^{-\lambda} \nonumber,
\end{align}

\end{lemma}

\begin{proof}
The proof of the lower bound goes in the same way in \cite[Lemma 7.1]{sta} by replacing $\frac{1}{\alpha}\log n$ to $\frac{1}{\alpha}\log n-\lambda$. Let $s:=\frac{1}{\alpha}\log n-\lambda$. Applying (\ref{manytoone}) and \eqref{S:4}, we get
\begin{align*}
\mathbf{P}\Big(\bigcup_{k=n+1}^{\alpha n}A_k^{(n,\lambda)}\Big)&\leq \sum_{k=n+1}^{\alpha n}\mathbf{E}
\Big(\sum_{|x|=k}\mathbf{1}_{\{V(x_i)\geq a_i^{(n,\lambda)},\, i\leq k, V(x)\leq s+K\}}\Big)\nonumber\\
&=\sum_{k=n+1}^{\alpha n}\mathbf{E}\Big(e^{S_k}\mathbf{1}_{\{S_i\geq a_i^{(n,\lambda)},\, i\leq k, S_k\leq s+K\}}\Big)\nonumber\\
&\leq\sum_{k=n+1}^{\alpha n}e^{s+K}\mathbf{P}\big(S_i\geq a_i^{(n,\lambda)}, i\leq k, S_k\leq s+K\big)\nonumber\\
&\leq c\sum_{k=n+1}^{\alpha n}e^{s+K} \frac{1}{n^{1+\frac{1}{\alpha}}} \nonumber\\
&\leq c e^{-\lambda},
\end{align*}
completing the proof.
\end{proof}
\qed

 Denote the natural filtration of the branching random walk by $(\mathcal{F}_n, n\geq0)$. Here we introduce the well-known change-of-probabilities setting in Lyons \cite{Ly} and spinal decomposition. With the nonnegative martingale $W_n$, we can define a new probability measure $\mathbf{Q}$ such that for any $n\geq1$,
\begin{align}
\mathbf{Q}\big|_{\mathcal{F}_n}:=W_n\cdot\mathbf{P}\big|_{\mathcal{F}_n},
\end{align}
where $\mathbf{Q}$ is defined on $\mathcal{F}_\infty(:=\!\!\vee_{n\geq0}\mathcal{F}_n)$.  Similarly we denote by $\mathbf{Q}_a $ the probability distribution associated to the branching random walk starting from $a$, and $\mathbf{E}_Q$ the corresponding expectation related $\mathbf{Q}(:=\mathbf{Q}_0)$.  Let us give a description of the branching random walk under $\mathbf{Q}$. We start from one single particle $\omega_0\!\!:=\!\!\varnothing$, located at
$V(\omega_0)=0$. At time $n+1$, each particle $\upsilon$ in the $n$th generation dies and gives birth to a point process independently distributed as $(V(x),|x|=1)$ under $\mathbf{P}_{V(\upsilon)}$ except one particle $\omega_n$, which dies and produces a point process distributed as $(V(x),|x|=1)$ under $\mathbf{Q}_{V(\omega_n)}$. While $\omega_{n+1}$ is chosen to be $\mu$  among the children of $\omega_n$, proportionally to $e^{-V(\mu)}$. Next we state the following fact about the spinal decomposition. \\

\textbf{Fact 7.1 (Lyons \cite{Ly})}.
Assume \eqref{stable0}. \\
(\romannumeral1) For any $|x|=n$, we have
\begin{align}
\mathbf{Q}(\omega_n=x|\mathcal{F}_n)=\frac{\,e^{-V(x)}}{W_n}.\nonumber
\end{align}
(\romannumeral2)
The spine process $(V(\omega_n))_{n\geq0}$ under $\mathbf{Q}$ has the distribution of $(S_n)_{n\geq0}$ (introduced in Section 2) under $\mathbf{P}$. \\
(\romannumeral3) Let
$\mathcal{G}_\infty:=\sigma\{\omega_j,V(\omega_j),\Omega(\omega_j), (V(u))_{u\,\in\,\Omega(\omega_j)},j\geq1\}  \nonumber$ be the $\sigma$-algebra of the spine and its brothers.  Denote by $\{\mu\nu,|\nu|\geq0\}$ the subtree of $\textbf{T}$ rooted at $\mu$. For any $\mu\in\Omega(\omega_k)$, the induced branching random walk $(V(\mu\nu),|\nu|\geq0)$ under $\mathbf{Q}$ and conditioned on $\mathcal{G}_\infty$ is distributed as $\mathbf{P}_{V(\mu)}$.

For $n\geq2$ and $0\leq\lambda\leq\frac{1}{2\alpha}\log n$, we define
\begin{align}
E(n,\lambda):=\bigcup_{k=n\!+\!1}^{\alpha n}(A_k^{(n,\lambda)}\cap B_k^{(n,\lambda)}).
\end{align}

\begin{lemma}\label{L:3.3}
Assume \eqref{stable0}, \eqref{stable1}--\eqref{stable3}. There exists  $c>0$ such that for any $n\geq2,\;0\leq\lambda\leq\frac{1}{2\alpha}\log n$, $m\geq4n$ and $0\leq \mu\leq \frac{1}{2\alpha}\log m$,
\begin{align}
\mathbf{P}\big(E(n,\lambda)\cap F(m,\mu)\big)\leq c\,e^{-\lambda-\mu}+c\,e^{-\mu}\frac{\log n}{n^{\frac{1}{\alpha}}} . \nonumber
\end{align}
\end{lemma}

\begin{proof}
For convenience, we write $s\!:=\!\frac{1}{\alpha}\log n\!-\!\lambda, \;t\!:=\!\frac{1}{\alpha}\log m\!-\!\mu$ .
\begin{align}
\mathbf{P}\big(E(n,\lambda)\cap F(m,\mu)\big)&\leq\mathbf{E}\Big(\mathbf{1}_{E(n,\lambda)}\sum_{k=m+1}^{\alpha m}\sum_{|x|=k}\mathbf{1}_{\{x\in A_k^{(m,~\!\mu)}\cap B_k^{(m,~\!\mu)}\}}\Big)\nonumber\\
&=\sum_{k=m+1}^{\alpha m} \mathbf{E_Q}\Big(\mathbf{1}_{E(n,\lambda)}e^{V(\omega_k)}\mathbf{1}_{\{\omega_k\in A_k^{(m,~\!\mu)}\cap B_k^{(m,~\!\mu)}\}}\Big)\nonumber\\
&\leq e^{t+K}\sum_{k=m+1}^{\alpha m}\sum_{l=n+1}^{\alpha n} \mathbf{E_Q}\Big(\sum_{|x|=l}
\mathbf{1}_{\{x\in A_l^{(n,~\!\lambda)}\cap B_l^{(n,~\!\lambda)},\omega_k\in A_k^{(m,~\!\mu)}\cap B_k^{(m,~\!\mu)}\}}\Big)\nonumber\\
&=:e^{t+K}\sum_{k=m+1}^{\alpha m}\sum_{l=n+1}^{\alpha n}I(k,l)\label{L:3.3_2}.
\end{align}
Decomposing the sum on the brothers of the spine, we obtain
\begin{align}
I(k,l)&=\mathbf{Q}(\omega_l\in A_l^{(n,~\!\lambda)}\cap B_l^{(n,~\!\lambda)}, \omega_k\in A_k^{(m,~\!\mu)}\cap B_k^{(m,~\!\mu)})\nonumber\\
& +\sum_{p=1}^l\mathbf{E_Q}\Big(\mathbf{1}_{\{\omega_k\in A_k^{(m,~\!\mu)}\cap B_k^{(m,~\!\mu)}\}}\sum_{x\in\Omega(\omega_p)}f_{k,l,p}\big(V(x)\big)\Big) \nonumber\\
&=:I_1{(k,l)}+\sum_{p=1}^lJ{(k,l,p)},
\end{align}
where $f_{k,l,p}(V(x)):=\mathbf{E_Q}\Big(\sum_{u\ge x, |u|=l}\mathbf{1}_{\{u\in
A_l^{(n,~\!\lambda)}\cap B_l^{(n,~\!\lambda)}\}}\big|\mathcal{G}_\infty\Big)$.  Recalling \textbf{Fact 7.1(\romannumeral3)},
we obtain,
\begin{align}
f_{k,l,p}(x)&\leq \mathbf{E}_{x}\Big(\sum_{|\nu|=l-p}\mathbf{1}_{\{V(\nu_i)\geq a_{i+p}^{(n,\lambda)},0\leq i\leq l-p, V(\nu)\leq s+K\}}\Big)\nonumber\\
& \leq e^{-x+s+K}\mathbf{P}_x(S_i\geq a_{i+p}^{(n,\lambda)}, 0\leq i\leq l-p, S_{l-p}\leq s+K) \label{L:3.3_1},
\end{align}
where the last step is from (\ref{manytoone}). To estimate $\sum_{p=1}^lJ{(k,l,p)} $, we break the sum into
two parts. Firstly consider the case $1\leq p\leq\frac{\alpha n}{4}$. By \eqref{S:4}, we have
\begin{align}
f_{k,l,p}(x)\leq ce^{-x+s+K}\frac{\;1+x_+}{n^{1+\frac{1}{\alpha}}}.\nonumber
\end{align}
Consequently,
\begin{align}
\sum_{p=1}^{\frac{\alpha n}{4}}J{(k,l,p)}&\leq ce^sn^{-1-\frac{1}{\alpha}}\sum_{p=1}^{\frac{\alpha n}{4}} \mathbf{E_Q}\Big(\mathbf{1}_{\{\omega_k\in A_k^{(m,~\!\mu)}\cap B_k^{(m,~\!\mu)}\}}\sum_{x\in\Omega(\omega_p)}(1+V(x)_+)e^{-V(x)}\Big)\nonumber\\
&\leq ce^sn^{-1-\frac{1}{\alpha}}\sum_{p=1}^{\frac{\alpha n}{4}} \mathbf{E_Q}
\Big(\mathbf{1}_{\{\omega_k\in A_k^{(m,~\!\mu)}\cap B_k^{(m,~\!\mu)}\}}e^{-(p-1)^{\frac{\gamma}{2}}}\Big), \nonumber
\end{align}
where the last inequality comes from the definition of $B_k^{(m,~\!\mu)}$. Note that by \eqref{S:4},
$$\mathbf{Q}(\omega_k\in A_k^{(m,~\!\mu)})=\mathbf{P}(t\leq S_k\leq t+K, S_i\ge a_i^{(m,\mu)}, 0\leq i\leq k)\leq c m^{-1-\frac{1}{\alpha}}$$ for all $m<k\leq \alpha m$.  It follows that
\begin{align}
\sum_{p=1}^{\frac{\alpha n}{4}}J{(k,l,p)}\leq ce^sn^{-1-\frac{1}{\alpha}}\mathbf{Q}(\omega_k\in A_k^{(m,\mu)})\leq ce^sn^{-1-\frac{1}{\alpha}}m^{-1-\frac{1}{\alpha}} \label{L:3.3a}.
\end{align}
On the other hand, when $\frac{\alpha n}{4}<p\leq l$, returning to \eqref{L:3.3_1},
\begin{align}
f_{k,l,p}(x)&\leq e^{-x+s+K}\mathbf{P}_x(S_i\geq a_{i+p}^{(n,\lambda)}, 0\leq i\leq l-p, S_{l-p}\leq s+K)\nonumber\\
&\leq ce^{-x+s+K} \frac{\;1\!+\!(x\!-\!s)_+}{(1\!+\!l\!-\!p)^{1\!+\!\frac{1}{\alpha}}} \nonumber,
\end{align}
which is from \eqref{S:3}. Hence
\begin{align}
&\sum_{\frac{\alpha n}{4}<p\leq l}J{(k,l,p)}\nonumber\\
&\leq c e^{s+K}\sum_{\frac{\alpha n}{4}<p\leq l}\frac{1}{\,(1\!+\!l\!-\!p)^{1\!+\!\frac{1}{\alpha}}}\mathbf{E_Q}\Big(\mathbf{1}_{\{\omega_k\in A_k^{(m,~\!\mu)}\cap B_k^{(m,~\!\mu)}\}}\sum_{x\in\Omega(\omega_p)}e^{-V(x)}{\;\big(1\!+\!(V(x)\!-\!s)_+\big)}\Big)\nonumber\\
&\leq c  e^s \sum_{\frac{\alpha n}{4}<p\leq l} \frac{  e^{-(p-1)^{\frac{\gamma}{2}}}} {\,(1\!+\!l\!-\!p)^{1\!+\!\frac{1}{\alpha}}}\mathbf{Q}(\omega_k\in A_k^{(m,~\!\mu)}). \nonumber
\end{align}
As a consequence,
\begin{align}
\sum_{\frac{\alpha n}{4}<p\leq l}J{(k,l,p)} \leq c {e^s} e^{-n^{\frac{\gamma}{3}}}m^{-1-\frac{1}{\alpha}}. \label{L:3.3b}
 \end{align}
It remains to estimate $I_1(k,l)$ for $n<l\le \alpha n<\frac{\alpha m}{4}<k\le \alpha m$. Clearly,
\begin{align}
I_1(k,l)&\leq \mathbf{Q}(\omega_l\in A_l^{(n,~\!\lambda)}, \omega_k\in A_k^{(m,~\!\mu)})\nonumber\\
&=\mathbf{P}(S_i\ge a_i^{(n,\lambda)},0\le i\le l, S_l\le s+K, S_j\ge a_j^{(m,\mu)}, 0\le j\le k, S_k\le t+K).\nonumber
\end{align}
We use the Markov property at $l$ and \eqref{S:4} to arrive at
\begin{align}
I_1(k,l)&\leq \frac{c}{(k-l)^{1+\frac{1}{\alpha}}}\mathbf{E}\big((1+S_l)\mathbf{1}_{\{S_i\ge a_i^{(n,\lambda)},0\le i\le l, S_l\le s+K\}}\big)\nonumber\\
&\leq c (1+s+K)(k-l)^{-1-\frac{1}{\alpha}}l^{-1-\frac{1}{\alpha}} \nonumber,
\end{align}
which together with \eqref{L:3.3a} and \eqref{L:3.3b} leads to
\begin{align}
I(k,l)\leq  ce^s(n^{-1-\frac{1}{\alpha}}m^{-1-\frac{1}{\alpha}}+e^{-n^{\frac{\gamma}{3}}}m^{-1-\frac{1}{\alpha}})+
c (1+s+K)(k-l)^{-1-\frac{1}{\alpha}}l^{-1-\frac{1}{\alpha}} \nonumber.
\end{align}
Recalling \eqref{L:3.3_2}, we have
\begin{align}
\mathbf{P}\big(E(n,\lambda)\cap F(m,\mu)\big)&\leq e^{t\!+\!K}\sum_{k=m+1}^{\alpha m}\sum_{l=n+1}^{\alpha n}
c\left(e^s(n^{\!-\!1\!-\!\frac{1}{\alpha}}m^{-1-\frac{1}{\alpha}})+\!(1\!+\!s\!+\!K)(k\!-\!l)^{\!-\!1\!-\!\frac{1}{\alpha}}l^{\!-\!1\!-\!\frac{1}{\alpha}} \right)\\&\leq c\,e^{-\lambda-\mu}+c\,e^{-\mu}\frac{\log n}{n^{\frac{1}{\alpha}}} .\nonumber
\end{align}
\qed
\end{proof}
\noindent\textbf{Proof of (\ref{div})}. Let $f$ be the nondecreasing function such that $\int_0^\infty \frac{dt}{te^{f(t)}}=\infty$. By Erd\"{o}s \cite{Erdos},
we can assume that $\frac{1}{2}\log (\log t)\leq f(t)\leq 2\log (\log t)$ for all large $t$ without any loss of generality.
Let
\begin{align}
  F_x :=\{M_n+x\leq \frac{1}{\alpha}\log n-f(n),\quad \mbox{i.o.}\}, \;x\in\mathbb{R} \nonumber.
\end{align}
We are going to prove that there exists $c_{8}>0$ such that for any $x$,
\begin{align}
\mathbf{P}( F_x)\ge c_{8}. \label{T:1.1_1}
\end{align}

Define $n_i=2^i$, $\lambda_i=f(n_{i+1})+x$ , and $ E_i=E(n_i,\lambda_i) $. It is easy to see for any $x\in \mathbb{R}$, we can choose  $i_0=i_0(x)$ such that $0\le\lambda_i\le \frac{1}{2\alpha}\log n_i$ for $i\ge i_0$. According to Lemma \ref{L:3.2} and Lemma \ref{L:3.3},  there exists $c>0$ such that for any $i\ge i_0, j\ge i+2$,
\begin{align}
& \frac{1}{c} e^{-\lambda_i}\leq \mathbf{P}(E_i)\le c e^{-\lambda_i} ,\;\;i\ge i_0 ,\nonumber\\
&\mathbf{P}(E_i\cap E_j)\le c e^{-\lambda_i-\lambda_j}+ce^{-\lambda_j}\frac{\log n_i}{n_i^\frac{1}{\alpha}} . \nonumber
\end{align}
It follows that
\begin{align}
&\sum_{i=i_0}^k\mathbf{P}(E_i)\ge c\sum_{i=i_0}^ke^{-\lambda_i},\nonumber\\
&\sum_{i,j=i_0}^k\mathbf{P}(E_i\cap E_j)\leq c\Big(\sum_{i=i_0}^ke^{-\lambda_i}\Big)^2+c\Big(\sum_{i=i_0}^ke^{-\lambda_i}\Big)
\Big(\sum_{i=1}^\infty\frac{\log n_i}{n_i^{\frac{1}{\alpha}}}\Big).
\end{align}
Note that $\sum_{i=i_0}e^{-\lambda_i}\ge c \sum_{i=i_0}e^{-f(n_{i+1})}\ge c\sum_{i=i_0}\int_{n_{i+1}}^{n_{i+2}}
\frac{dt}{te^{f(t)}}=\infty$. Thus, we can find a constant $c_8$ (notice that our choice of $c_8$ does not depend on $x$) such that
\begin{align}
\limsup_{k\rightarrow\infty}\frac{\sum_{i,j=1}^k\mathbf{P}(E_i\cap E_j)}{(\sum_{i=1}^k\mathbf{P}(E_i))^2}\leq c_8.\nonumber
\end{align}
By Kochen and Stone's version of Borel-Cantelli Lemma \cite{B_C}, we have $\mathbf{P}(E_i,\;i.o.)\ge c_8$, which implies \eqref{T:1.1_1}. Let $F_\infty:=\cap_{x=1}^\infty F_x$. We have $\mathbf{P}(F_\infty)\ge c_8$ since $F_x$ are non-increasing on $x$. We then use the branching property to obtain
\begin{align}
\mathbf{P}(F_\infty|\mathcal{F}_k)=\mathbf{1}_{\{Z_k>0\}}(1-\prod_{|x|=k}(1-\mathbf{P}_{V(x)}(F_\infty)))
\leq \mathbf{1}_{\{Z_k>0\}} (1-(1-c_8)^{Z_k}).
\end{align}
Leting $k\rightarrow\infty$ in the above inequality, we conclude that
\begin{align}
\mathbf{1}_{F_\infty}=\mathbf{1}_{\{\mbox{non-extinction}\}}\;\;\mathbf{P}-a.s.
\end{align}
The divergence part (\ref{div}) is now proved.
\qed

\noindent\textbf{Proofs of Theorems 1.1--1.2.}  They are immediate by combining the above proofs of (\ref{conver}) and (\ref{div}).\qed

\noindent\textbf{Proof of Theorem 1.3} In Theorem~\ref{T:1.2}, taking $f(n)=\log\log n$ and $(1+\varepsilon)\log\log n$ for $\varepsilon>0$, we obtain the desired result. \qed

\section{Proof of Theorem 1.4}
\begin{lemma}\label{L:4.1}
Assume \eqref{stable0}, \eqref{stable1}--\eqref{stable3}. For any $\lambda>0$, there is $c_9>0$ such that for each $n\geq1$,
\begin{align}
\mathbf{P}\Big(M_n<\big(1+\frac{1}{\alpha}\big)\log n-\lambda\Big)\leq c_9 (1+\lambda)e^{-\lambda}. \nonumber
\end{align}
\end{lemma}

\begin{proof}
  If we have proved that for any $\lambda,\beta>0$, there exists   $c$ such that for any $n\geq1$, 
\begin{align}
\mathbf{P}\Big(&M_n<\big(1+\frac{1}{\alpha}\big)\log n-\lambda, \min_{|u|\leq n}V(u)\geq-\beta\Big)\nonumber\\ &\leq c(1+\beta)e^{-\lambda}\bigg(1+\frac{\Big(1+\big(\beta+(1+\frac{1}{\alpha})\log n-\lambda\big)_+\Big)^{2\alpha+1}}{n^{\frac{1}{\alpha}}}\bigg), \label{E:4.1}
\end{align}
then by the following fact
\begin{align}
\mathbf{P}(\inf_{u\in \textbf{T}}V(u)<-\lambda)\leq e^{-\lambda}, \nonumber
\end{align}
we can obtain the proof of Lemma~\ref{L:4.1}.

 Now we turn to prove \eqref{E:4.1}. For brevity we write $b=\big(1+\frac{1}{\alpha}\big)\log n-\!\lambda-\!1$. Note that we can assume $b+1>-\beta$, otherwise there is nothing to prove for \eqref{E:4.1}. For  $|u|=n$ such that $V(u)<b+1$, either $\min_{\frac{n}{2}\leq j\leq n}V(u_j)>b$, or $\min_{\frac{n}{2}\leq j\leq n}V(u_j)\leq b$. For the latter case, we shall consider the first $j\in [\frac{n}{2}, n]$ such that $V(u_j)\leq b$. Then
\begin{align}
\mathbf{P}\Big(M_n<\big(1+\frac{1}{\alpha}\big)\log n-\lambda, \min_{|u|\leq n}V(u)\geq -\beta\Big)\leq \mathbf{P}( H_1 )+\mathbf{P}( H_2 ). \label{a0}
\end{align}
with
\begin{align}
& H_1 :=\Big\{\exists |u|=n: V(u)<b+1, \underline{V}(u)\geq -\beta, \min_{\frac{n}{2}\leq j\leq n}V(u_j)>b \Big\}, \nonumber\\
& H_2 :=\bigcup_{\frac{n}{2}\leq j\leq n} \Big\{\exists |u|=n: V(u)<b+1, \underline{V}(u)\geq -\beta, \min_{\frac{n}{2}\leq i< j}V(u_i)>b, V(u_j)\leq b \Big\} \nonumber.
\end{align}
By (\ref{manytoone}) and (\ref{S:4}), we have
\begin{align}
\mathbf{P}(H_1)\leq &\;\mathbf{E}\,\Big(\sum_{|u|=n}\textbf{1}_{\{V(u)<b+1,\, \underline{V}(u)\geq -\beta, \, \min_{\frac{n}{2}\leq j\leq n}V(u_j)>b\}}\Big)\nonumber\\
=&\;\mathbf{E}\,\Big(e^{S_n}\textbf{1}_{\{S_n<b+1,\, \underline{S}_n\geq-\beta, \,\min_{\frac{n}{2}\leq j\leq n}S_j>b\}}\Big) \nonumber\\
\leq &\;   c \,e^b(1+\beta)\,n^{-1-\frac{1}{\alpha}} \nonumber\\
\leq & \; c \,(1+\beta)\,e^{-\lambda}.\label{a1}
\end{align}

To deal with $\mathbf{P}(H_2)$, we consider $v=u_j$ and use the notation $  |u|_v:=|\mu|-|\nu|=n-j $ and $V_v(u):=V(u)-V(v)$ for $|u|=n$ and $v<u$. Then by the Markov property,
\begin{align}
&\mathbf{P}(H_2)  \nonumber\\
\leq & \sum_{\frac{n}{2}\leq j\leq n} \mathbf{E}\Big(\sum_{|v|=j}\textbf{1}_{\{\underline{V}(v)\geq -\beta, \min_{\frac{n}{2}\leq i<j}V(v_i)>b, V(v)\leq b\}}\!\!\sum_{|u|_v=n-j}\textbf{1}_{\{V_v(u)\leq b+1-V(v), \min_{j\leq i\leq n}V_v(u_i)\geq-\beta-V(v)\}}\Big) \nonumber\\
=&\sum_{\frac{n}{2}\leq j\leq n}\mathbf{E}\Big(\sum_{|v|=j}\textbf{1}_{\{\underline{V}(v)\geq -\beta, \min_{\frac{n}{2}\leq i<j}V(v_i)>b, V(v)\leq b\}}\phi(V(v),n-j)\Big) \label{a666}\\
=:&\mathbf{ E_\eqref{a666}} + \mathbf{E'_\eqref{a666}}  \label{a2}
\end{align}
where  $E_\eqref{a666}$  denotes the sum $\sum_{\frac{n}{2}\leq j\leq \frac{3n}{4}}$ and $E'_\eqref{a666}$ denotes the sum $\sum_{\frac{3n}{4}<j\leq n}$ in \eqref{a2}, and
\begin{align}
\phi(x, n-j):=&\mathbf{E}\Big(\sum_{|u|_v=n-j}\textbf{1}_{\{V_v(u)\leq b+1-V(v), \min_{j\leq i\leq n}V_v(u_i)\geq-\beta-V(v)\}}\Big|V(v)=x\Big)\nonumber\\
=&\mathbf{E}\Big(e^{S_{n-j}}\textbf{1}_{\{S_{n-j}\leq b+1-x,\underline{S}_{n-j}\geq-\beta-x\}}\Big).\nonumber
\end{align}
It  follows from \eqref{S:3} that
\begin{align}
\phi(x,n-j)\leq {\blue c}(1+\beta+x)(2+\beta+b)^\alpha(n-j+1)^{-1-\frac{1}{\alpha}}e^{b-x}. \label{E:4.4}
\end{align}
By  \eqref{E:4.4}, (\ref{manytoone}) and then (\ref{S:3}), we obtain that
\begin{align}
  E_\eqref{a666} \leq& \;{\blue c}\sum_{\frac{n}{2}\leq j\leq \frac{3n}{4}}(2+b+\beta)^\alpha n^{-1-\frac{1}{\alpha}}e^b\mathbf{E}\Big((1+\beta+S_j)\textbf{1}_{\{\underline{S}_j\geq-\beta, \min_{\frac{n}{2}\leq i<j}S_i>b, S_j\leq b\}}\Big)\nonumber\\
\leq& \;  c \,(2+b+\beta)^{\alpha+1}e^{-\lambda}\sum_{\frac{n}{2}\leq j\leq \frac{3n}{4}}\mathbf{P}(\underline{S}_j\geq-\beta,S_j\leq b)\nonumber\\
\leq&\;  c \,(1+\beta)(2+b+\beta)^{2\alpha+1}e^{-\lambda}\sum_{\frac{n}{2}\leq j\leq \frac{3n}{4}}j^{-1-\frac{1}{\alpha}}  \nonumber\\
\leq& \;  c \,(1+\beta)\frac{(1+\beta+(1+\frac{1}{\alpha})\log n-\lambda)^{2\alpha+1}}{n^{\frac{1}{\alpha}}}e^{-\lambda}. \label{a21}
\end{align}
Meanwhile, by the estimate $\phi(x,n-j)\leq e^{b+1-x}$, we get that
\begin{align}
 E'_\eqref{a666} \leq &\sum_{\frac{3n}{4}\leq j\leq n}\mathbf{E}\Big(\sum_{|v|=j}\textbf{1}_{\{\underline{V}(v)\geq-\beta, \min_{\frac{n}{2}\leq i<j}V(v_i)>b, V(v)\leq b\}}e^{b+1-V(v)}\Big)\nonumber\\
=& \;e^{b+1}\sum_{\frac{3n}{4}\leq j\leq n}\mathbf{P}\Big(\underline{S}_j\geq-\beta, \min_{\frac{n}{2}\leq i<j}S_i>b, S_j\leq b\Big)\nonumber\\
\leq &\,  c \,e^b(1+\beta)\,n^{-1-\frac{1}{\alpha}}\nonumber\\
\leq &\,  c \,(1+\beta)e^{-\lambda}.\label{a22}
\end{align}

Combing the estimates  (\ref{a0})--(\ref{a22}), we get \eqref{E:4.1}, and then complete the proof. \qed
\end{proof}

\noindent\textbf{Proof of Theorem \ref{T:1.4}}.
Consider large integer $j$. Let $n_j:=2^j$ and $\lambda_j:=a\log\log\log n_j$ with some constant $0<a<1$. Put
\begin{align}
  K_j :=\big\{M_{n_j}>(1+\frac{1}{\alpha})\log n_j+\lambda_j\big\}. \nonumber
\end{align}
Recall that if the system dies out at generation $n_j$, then  $M_{n_j}=\infty$. Define
$M_.^{(u)}$ for the subtree $\mathbb{T}_u$ just as $M_.$ for $\mathbb{T}$. Then   $$  K_j =\{ |u|=n_{j-1}, M_{n_j-n_{j-1}}^{(u)}> (1+\frac{1}{\alpha})\log n_j+\lambda_j-V(u)\}.$$  By the branching property at $n_{j-1}$ we obtain
\begin{align}
\mathbf{P^*}\big(  K_j \big|\mathcal{F}_{n_{j-1}}\big)=\prod_{|u|=n_{j-1}}\mathbf{P^*}\Big(M_{n_j-n_{j-1}}\geq
(1+\frac{1}{\alpha})\log n_j+\lambda_j-x\Big)\Big|_{x=V(u)}.\nonumber
\end{align}
By  Theorem \ref{T:1.3}, a.s. for all large $j$, $M_{n_{j-1}}\geq \frac{1}{2\alpha}\log n_{j-1}\sim cj$, hence $x\equiv V(u)\gg\lambda_j$ since $\lambda_j\sim a \log\log j$. By Lemma 4.1, on
$\{M_{n_{j-1}}\geq\frac{1}{2\alpha}\log n_{j-1}\}$, for some constant {\blue $c>0$} and all $|u|=n_{j-1}$,
\begin{align}
\mathbf{P^*}\Big(M_{n_j-n_{j-1}}<(1+\frac{1}{\alpha})\log n_j+\lambda_j-x\Big)\Big|_{x=V(u)}\leq {\blue c} V(u)e^{-(V(u)-\lambda_j)} .\nonumber
\end{align}
For sufficiently large $j$,  it follows that
\begin{align}
\mathbf{P^*}\big(K_j\big|\mathcal{F}_{n_{j-1}}\big)\geq &\textbf{1}_{\{M_{n_{j-1}}\geq \frac{1}{\,2\alpha}\log n_{j-1}\}}
\prod_{|u|=n_{j-1}}\Big(1-{\blue c} V(u)e^{-(V(u)-\lambda_j)}\Big)\nonumber\\
\geq &\textbf{1}_{\{M_{n_{j-1}}\geq \frac{1}{\,2\alpha}\log n_{j-1}\}} \exp\Big(-2{\blue c} \sum_{|u|=n_{j-1}}V(u)e^{-(V(u)-\lambda_j)}\Big)\nonumber\\
= &\textbf{1}_{\{M_{n_{j-1}}\geq \frac{1}{\,2\alpha}\log n_{j-1}\}} \exp\Big(-2{\blue c} e^{\lambda_j}D_{n_{j-1}}\Big).\nonumber
\end{align}
By \cite[Theorem 1.1]{sta}, $D_{n_{j-1}}\rightarrow D_\infty, \;a.s.$ Recalling $e^{\lambda_j}\sim(\log j)^a$ with $a<1$, we get that
\begin{align}
\sum_j\mathbf{P^*}\big(K_j\big|\mathcal{F}_{n_{j-1}}\big)=\infty, \quad a.s.\nonumber
\end{align}
which according to L$\acute{e}$vy's conditional form of Borel-Cantelli's lemma (\cite[Corollary 68]{Borel}), implies that $\mathbf{P^*}(  K_i, \mbox{i.o.})=1$. Thus
\begin{align}
\limsup_{n\rightarrow\infty}\frac{M_n-(1+\frac{1}{\alpha})\log n}{\log\log\log n}\geq a, \; \mathbf{P^*}-a.s.\nonumber
\end{align}
The proof is completed by letting $a\rightarrow 1$.
\qed


\begin{thebibliography}{99}

\bibitem{AR09} Addario-Berry, L. and Reed, B. (2009): Minima in branching random walks. \emph{Ann. Probab.} \textbf{37}, 1044--1079.

\bibitem{A13} Aid$\mathrm{\acute{e}}$kon, E. (2013): Convergence in law of the minimum of a branching random walk. \emph{Ann. Probab.} \textbf{41}, 1362--1426.

\bibitem{AS14} Aidekon, E. and Shi, Z. (2014): The Seneta-Heyde scaling for the branching random walk. \emph{Ann. Probab.} \textbf{42}, 959--993.

\bibitem{bi} Biggins, J. D. (2010): Branching out. In \emph{Probability and Mathematical Genetics}
(N. H. Bingham and C. M. Goldie, eds.). \emph{London Mathematical Society Lecture Note Series}
\textbf{378}, 113--134. Cambridge Univ. Press, Cambridge.

\bibitem{B77} Biggins, J. D. (1977): Martingale convergence in the branching random walk. \emph{J. Appl. Probab.}
\textbf{14}, 25--37.

\bibitem{bi2} Biggins, J. D. (1976): The first- and last-birth problems for a multitype age-dependent branching process. \emph{Avd. Appl. Probab.} \textbf{8}, 446--459.

\bibitem{BK05} Biggins, J. D. and Kyprianou, A. E. (2005): Fixed points of the smoothing transform: The
boundary case. \emph{Electron. J. Probab.} \textbf{10}, 609--631.

\bibitem{BZ06} Bramson, M.D. and Zeitouni, O. (2006): Tightness for a family of recursion equations. \emph{Ann. Probab.} \textbf{37}, 615¡ª653.

\bibitem{Erdos} Erd\"{o}s, P. (1942): On the law of the iterated logarithm. \emph{Ann. Math.} (2)\textbf{43}, 419--436.

\bibitem{sta} He, H. Liu, J. and Zhang, M (2016): On Seneta-Heyde Scaling for a stable branching random walk. https://arxiv.org/abs/1610.03575.

\bibitem{H15} Hu, Y. (2015): The almost sure limits of the minimal position and the additive martingale in a branching random walk. \emph{J. Theor. Probab.} \textbf{28}, 467¡ª487.

\bibitem{H13} Hu, Y. (2013): How big is the minimum of a branching random walk? \emph{Ann. De I Ins. Henri Poincare Probab. et. Stat.} \textbf{52}(1).

\bibitem{HS09} Hu, Y. and Shi, Z. (2009): Minimal position and critical martingale convergence in branching random walks, and directed polymers on disordered trees. \emph{Ann. Probab.} \textbf{37}, 742--789.

\bibitem{Ja} Jaffuel, B. (2012): The critical barrier for the survival of the branching random walk with absorption. \emph{Ann. Inst. H. Poincar$\acute{e}$ Probab. Statist.} \textbf{48} 989--1009.

\bibitem{B_C} Kochen, S. and Stone, C. (1964): A note on the Borel-Cantelli lemma. I\!I\!I \emph{J. Math.} \textbf{8}, 248--251.

\bibitem{Borel} L$\acute{e}$vy, P. (1937): \emph{Th$\acute{e}$orie de l'addition des variables al$\acute{e}$atoires}. Gauthier-Villars, Paris.

\bibitem{Ly} Lyons, R. (1997): A simple path to Biggins' martingale convergence for branching random walk. In: \emph{Classical and Modern Branching Processes} (Eds.: Athreya, K.B. and Jagers, P.) \emph{IMA Volumes in Mathematics and its Applications} \textbf{84}, 217--221. Springer, New York.


\end{thebibliography}
\end{document}